\documentclass{amsart}

\usepackage{ae}
\usepackage{graphicx}
\usepackage{epsfig}
\usepackage{graphics}

\theoremstyle{plain}

\newtheorem{theorem}{Theorem}
\newtheorem{corollary}[theorem]{Corollary}
\newtheorem{prop}[theorem]{Proposition}

\theoremstyle{definition}

\newtheorem{definition}{Definition}

\newtheorem*{open}{Open Problems}

\newcommand{\PP}{\mathcal{P}}
\newcommand{\N}{\mathcal{N}}
\newcommand{\G}{\mathcal{G}}
\newcommand{\cclass}[1]{\ensuremath{\mathord{\rm #1}}} 

\newcommand{\invisibleComment}[1]{}

\DeclareMathOperator{\Mex}{Mex}

\title{Impartial coloring games}

\author{G. Beaulieu\\
{\tiny Universit de Sherbrooke\\
Ssherbrooke, J1K 2R1, Quebec.\\
gabriel.beaulieu@usherbrooke.ca}\\
K. Burke\\
{\tiny Wittenberg University\\
Springfield, OH 45503, USA.\\
kburke@wittenberg.edu}\\
E. Duch\^{e}ne\\
{\tiny Laboratoire GAMA,\\
Universit\'{e} Claude Bernard Lyon 1, Universit de Lyon, F-69622 France.\\
eric.duchene@univ-lyon1.fr}
}

\begin{document}
\maketitle
\begin{abstract}
 Coloring games are combinatorial games where the players alternate painting uncolored vertices of a graph one of $k > 0$ colors.  Each different ruleset specifies that game's coloring constraints.  This paper investigates six impartial rulesets (five new), derived from previously-studied graph coloring schemes, including proper map coloring, oriented coloring, 2-distance coloring, weak coloring, and sequential coloring.  For each, we study the outcome classes for special cases and general computational complexity.  In some cases we pay special attention to the Grundy function.
\end{abstract}

\section{Introduction and Background}
	
This paper deals with impartial coloring games on graphs. Roughly speaking, two players alternate coloring the vertices of a graph (under some specific constraints), until no more vertices can be painted. Following the normal play convention, the first player unable to color a vertex loses the game \cite{Win}.  In each of these rulesets, a description of the game state consists of a graph, a number of colors and a partial coloring of the vertices, adhering to the coloring restrictions of the ruleset.  The options for each position are those where exactly one new vertex has been colored.

\begin{definition}
  An \emph{impartial coloring ruleset} is an impartial combinatorial game ruleset where players alternate coloring the vertices of a graph.  The state of the game consists of: a graph, $G = (V, E)$, which may be either directed or undirected;  a positive integer, $k$, representing the number of colors; and a partial coloring of the vertices, $c: V \rightarrow \{1, \ldots, k\} \cup \{\text{`uncolored'}\}$.  Impartial coloring rulesets include some sort of restriction to which $c$ must adhere.  The options for each position are those which have the same state aside from a new coloring, $c'$ where $V = V'\cup \{v_0\}, c'(V') = c(V'), c'(v_0) \neq c(v_0)$ and $c(v_0) = \text{`uncolored'}$.
\end{definition}

We define many such rulesets, analyze outcome classes for different situations, and study the computational complexity of each.  In section \ref{properColoring}, we study the rules for \textsc{Proper $k$-coloring}, based on proper map-colorings.  In section \ref{orientedColoring}, we use oriented colorings on digraphs \cite{Sop} to define the games \textsc{Oriented $k$-coloring} and \textsc{Oriented Blue-Red-Coloring}.  In section, \ref{weakColoring}, we base \textsc{Weak $k$-coloring} on weak colorings, where each non-isolated vertex must be adjacent to a vertex of at least one other color \cite{Chew}.  In section, \ref{distanceColoring}, we use the $2$-distance coloring defined in \cite{Fer} to define the ruleset \textsc{2-distance Coloring}.  The last game, \textsc{Sequential Coloring}, defined first in \cite{Bod}, is covered in section \ref{sequentialColoring}, including an improvement to the original solution algorithm on paths and cycles.

The rest of this section is devoted to basic combinatorial game theory definitions, with a focus on impartial game theory.  Those who already have an understanding of the topic (known as ``gamesters'') may safely skip to the next section.  

\subsection{Combinatorial Game Theory Basics}

These standard definitions can be found in \cite{Win}. The set of nonnegative integers will be denoted by $\mathbb{N}$. For all $U\subset \mathbb{N}$, $\Mex(U)$ denotes the smallest non-negative integer not included in $U$. We will denote by $\mathbb{G}$ the set of positions of a given game. If $\textit{u}\in \mathbb{G}$, then $O(\textit{u})$ denotes the {\it options} of $u$, i.e., the set of positions reachable by a legal move from \textit{u}.\\

In combinatorial game theory, \emph{partisan} games are those for which both players do not always have the same options from a position. The investigation of partisan coloring games on graphs led to the map-coloring game (introduced by M. Gardner in 1981), and the corresponding game chromatic number \cite{Bar,gcn}. One can also refer to the game of {\it Col}, which is a partisan coloring game on planar graphs \cite{Win}. 

Alternatively, impartial games are those for which both players always have the same options.  Notably, the impartial game Nim was completely solved \cite{Nim} in 1905. In the case of impartial coloring games, there are few results in the literature. In the 90's, Bodlaender investigated online (or sequential) coloring games \cite{Bod,Onli,Onli2}. One can also consider \textsc{Node-Kayles} \cite{kayles} as an impartial coloring game with a unique color. In \cite{Bod}, it is shown that finding a winning strategy for some coloring games, such as the sequential coloring game, or \textsc{Node-Kayles} \cite{kayles} are \cclass{PSPACE}-complete. \invisibleComment{Sequential Coloring can be solved in poly-time, no? Yes, for 2 colors. But not in the general case} \\

Given a short\invisibleComment{Kyle: Made the change.  Yes, short games are those with bounded depth.  I believe this is standard notation in Lessons in Play. Eric: Ok Kyle I trust you. I don't have this book} impartial combinatorial game (i.e., an impartial game where any position cannot be visited twice), positions are classified as $\PP$ or $\N$: it is a $\PP$-position if the second player has a winning strategy, and an $\N$-position otherwise (the first player has a winning strategy). In addition, Sprague and Grundy introduced a function $\G:\mathbb{G}\rightarrow \mathbb{N}$ called the Grundy function, such that $\forall\textit{u},\textit{v} \in \mathbb{G}$
\begin{itemize}
	\item $\G(u)=0$ if and only if $u$ is a $\PP$-position,
	\item $\G(u)>0$ if and only if $u$ is a $\N$-position,
	\item $\G(u)=\G(v)$ if and only if playing on the sum $u+v$ is a $\PP$-position.
\end{itemize}

According to Sprague and Grundy \cite{Spr}, it turns out that $\forall$\textit{u}, $\G(u)=Mex(\G(O(u))$.\\


In terms of complexity, our objective is also to have a better idea of what can be expected for each of these games. Therefore, we will mainly focus our research on simple families of graphs such as paths and cycles.  As much as possible, we determine the computational complexity of these games, evidence of the difficulty for players to choose optimal moves on their turn.  (See \cite{Dem} for more information on computational complexity and its relevance to combinatorial game theory.).\invisibleComment{Added :)}

	
\section{Proper k-colorings}
\label{properColoring}

\begin{definition}
	\label{defjeucolor}
\textsc{Proper $k$-coloring} is an impartial coloring ruleset where no two adjacent vertices of the graph may be painted the same color.  
\end{definition}

When $k=1$, the game is known as \textsc{Node-Kayles} \cite{Win}, which is known to be \cclass{PSPACE}-complete \cite{Schae2}. As we show next, the general game is also \cclass{PSPACE}-complete for all $k$.

\begin{prop}
  \label{properPSPACE}
  Determining whether an instance of {\sc Proper $k$-coloring} is an $\N$-position is \cclass{PSPACE}-complete.
\end{prop}

\begin{proof}
  Since all instances of the game are short and the number of plays is bounded above by the number of unpainted vertices, the problem is in \cclass{PSPACE}.  To show hardness, we will reduce from \textsc{Node-Kayles}.  The reduction is straightforward; we simply connect each original Kayles vertex to new vertices, each painted one of the $k-1$ colors.

  Given an instance of \textsc{Node-Kayles}, $G = (V, E)$, create a new graph, $G' = (V', E')$ where 

$V' = \{v_0, \ldots , v_{k-1} : \forall v \in V \}$, and 

$E' = \{(v_0, v_i) : \forall v \in V, i \in \{1, \ldots, k-1\}\} \cup \{(u_0, w_0) : \forall (u, w) \in E \}$.

  \begin{figure}[!ht]
	  \label{fig:properKColoringReduction}
	  \centering
		  \includegraphics[width=60mm]{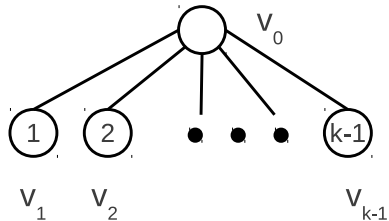}
	  \caption{Reduction of vertex, $v$, from \textsc{Node-Kayles} to the \textsc{Proper $k$-coloring gadget}.  $v_0$ is uncolored.}
  \end{figure}

  To complete the reduction, we paint all vertices $v_i$ where $i > 0$ with the $i^{th}$ color.  Now, only the vertices $v_0$ remain unpainted, all with only a single color available.  These are adjacent to neighbors just as their counterparts, $v$ are.  Thus, each move in this game is equivalent to the corresponding move on the \textsc{Node-Kayles} instance.  So, this game is also \cclass{PSPACE}-hard.  Note that this reduction preserves the planarity of the original graph.

  Together, determining whether a position of {\sc proper $k$-coloring} is in outcome class $\N$ is \cclass{PSPACE}-complete.
\end{proof}

The following proposition provides a strong symmetry argument to solve the {\sc proper k-coloring game} on several families of graphs.

\begin{prop}
	\label{gagnant}
	Let $G=(V,E)$ be a graph. If there exists an involution, \textit{s}, of $G$ satisfying:
	\begin{itemize}
		\item $\exists$! u $\in$V such that \textit{s}(u)=u
		\item $\forall$ v $\in V$, (v,\textit{s}(v)) $\not\in$ E,
	\end{itemize}
	then the starting position of \textsc{Proper $k$-coloring} on $G$ is in $\N$.
\end{prop}

\begin{proof}
Let us play the first move on the unique vertex \textit{u} satisfying $\textit{s}(u)=u$. We will recursively show that the first player is able to hold the following property $(P1)$ during the game: \\

$(P1)$ : For all vertex $v\in V\setminus \{u\}$, the vertices $v$ and $s(v)$ hold the same color before second player's turn. \\

The property $(P1)$ is true after the first move on $u$. Assume $(P1)$ is true after any odd number of moves.
	Without loss of generality, suppose that the second player paints some vertex $v_{0}\in V$ with the color $i\in \lbrace$1...k$\rbrace$. Since $(P1)$ was true before this move, then $s(v_{0})$ is not colored yet. Hence the first player chooses to color $s(v_{0})$ with the color $i$. This is allowed since $s$ is an involution: indeed, suppose there exists some vertex $v_{2}=s(v_{1}$), adjacent to $s(v_{0})$, and colored with $i$. Since $(P1)$ was true before the second player's move, $v_{1}$ is also colored with $i$. But since $v_{2}=s(v_{1})$ is adjacent to $s(v_{0})$, the vertex $v_{1}$ is adjacent to $v_{0}$, which yields a contradiction with the proper coloring constraint \invisibleComment{Which part does this yield a contradiction to, exactly? @Kyle:is it better?  Kyle: Much better!} Consequently we can color $s(v_{0})$ with $i$, and then $(P1)$ holds. \\
	
	Thus, the first player always has a legal move after the second player's turn. 
\end{proof}

For $k=2$, Proposition \ref{gagnant} has a similar result for deciding whether a game position is an $\N$-position.

\begin{prop}
	\label{perdant}
	Let G=(V,E) be a graph.  If there exists an involution, \textit{s}, without a fixed point, the starting position of {\sc Proper 2-coloring} with colors $\{R,B\}$ is in $\PP$.
\end{prop}

\begin{proof}
	We will use a similar proof to Proposition \ref{gagnant}, by showing that the second player is able to hold the following property $(P2)$ during the game :\\

$(P2)$ : After each even move, for each colored vertex $v$, the image, $s(v)$, is painted the opposite color.\\

Trivially, $(P2)$ is true before the first move. Inductively, assume $(P2)$ is true after any even number of moves. Now suppose, without loss of generality, that the first player paints some vertex $v_{0}$ with the color $B$. Hence we can color $s(v_{0})$ with $R$: suppose indeed there exists some vertex $v_{2}=s(v_{1})$, adjacent to $s(v_{0})$, and colored with $R$. As $(P2)$ was true before the first player's move, $v_{1}$ is colored with $B$, and since $s$ is an involution, the vertices $v_{1}$ and $v_{0}$ are adjacent, and coloring $v_{0}$ with $B$ was not a legal move \invisibleComment{is this the contradiction? @Kyle:yes it is. I understand your remark, the sentence could be improved. Kyle: I think it's okay after rereading it.}. Therefore we can color $s(v_{0})$ with $R$.
\end{proof}

These two propositions allow us to come to further conclusions about the status of the game positions on several families of graphs. For example, when playing on paths with two colors, we have the following result:

\begin{corollary}
	\label{paths}
	{\sc Proper 2-coloring} starting positions are $\N$ positions when played on paths of odd length, and $\PP$ positions on paths of even length.
\end{corollary}

\begin{proof}
	In the case of paths of odd length, let us label the vertices from $1$ to $2k-1$ for any positive integer $k$. Let $s$ be a function on the vertex, such that $s(i)=2k-i$. It is easy to verify that $s$ is an involution having a unique fixed point, with no edge between any vertex $u$ and $s(u)$. Hence, according to Proposition \ref{gagnant}, the position is a first-player win.\\
	
In the case of paths of even length with vertices labeled from $1$ to $2k$, we consider the function $s$ defined on $V$ such that $s(i)=2k-i+1$. It is easy to verify that $s$ is an involution without a fixed point. Hence, according to Theorem \ref{perdant}, the game is a loss for the first player.
\end{proof}

\begin{corollary}
	{\sc Proper 2-coloring} starting positions are in $\PP$ when played on cycles.
\end{corollary}

\begin{proof}
	For cycles of even length, the proof is similar as in Corollary \ref{paths}. In the case of odd cycles, Propositions \ref{gagnant} and \ref{perdant} cannot be applied. We thus prove differently that there is a winning strategy for the second player:
Let $v_{0}$ be the first vertex colored (say with the color $B$) by the first player. Without loss of generality, let $v_{1}$ be an adjacent vertex of $v_0$ and $v_{2}$ the other adjacent vertex of $v_{1}$. The second player then colors $v_{2}$ with $R$. Hence, $v_{1}$ cannot be colored any more.
	As $v_{1}$ can not be colored, we can delete it from the graph. The second player can now apply the strategy given in the proof of Proposition \ref{gagnant} as if they played on a path of even length with both extremities having opposite colors. 
\end{proof}

By the same methods, Propositions \ref{gagnant} and \ref{perdant} solve the proper coloring game for other classes of graphs. Given any positive integer $k$, the {\sc proper k-coloring} starting positions are in $\N$ in the following cases:
	\begin{itemize}
		\item grids of size $n\times m$, when $n\times m$ is odd
		\item $k$-dimensional grids, when all dimensions are odd
		\item complete binary trees
		\item $k$-complete trees when $k$ is odd
		\item odd-length paths
	\end{itemize}
	By Proposition \ref{perdant}, the following classes of graphs yield {\sc Proper 2-coloring} starting positions in $\PP$:
	\begin{itemize}
		\item hypercubes
		\item grids of size $n\times m$, when $n\times m$ is even
		\item $k$-dimensional grids, with at least one even dimension
		\item paths and cycles of even length
	\end{itemize}

In addition, the Grundy function is fully determined in the case of paths and cycles. The only non-trivial case is for paths of odd length, which is detailed below.

\begin{prop} 	
For any starting position, $G$, of {\sc proper coloring} on a path of odd length, $\G(G)=1$.
\end{prop}

	
\begin{proof}
	Let us consider a coloring game $(G1)$ on a path of odd length $2k+1$, added to a coloring game $(G2)$ on a path of length 1. We label the vertices of the path on $(G1)$ from $1$ to $2k+1$. The game sum $G1+G2$ is a losing game. Consider the following nearly-symmetric strategy for the second player:
	\begin{itemize}
		\item If the first player plays on node $k+1$, then color the node of G2
		\item If the first player plays on G2, color the node $k+1$
		\item If the first player plays on node $i$, play on $2(k+1)-i$ with the same color
	\end{itemize}
	
	Now we must ensure that the second player always has a move. It is clear that in the first case, player 2 can play. In the second case, node $k+1$ is not colored yet (because G2 was not yet colored). We can color node $k+1$ because if any adjacent node $i$ is colored ($k$ or $k+2$), the node $2(k+1)-i$ is colored in the same color, and $k+1$ can be colored. In the last case, the coloration of the two sides of the paths (from $1$ to $k$ and $2k+1$ to $k+2$) are similar due to strategy property, and then if the first players can play, the second player can maintain the similarity and play too.
\end{proof}
\begin{open}
Despite these results, there are simple families of graphs which remain hard to solve, such as caterpillars (i.e., trees in which all the vertices are within distance $1$ of a central path), or $k$-complete trees when $k$ is even. Moreover, deciding whether there exists an involution in a graph is a hard problem in general. 
\end{open}

\section{Oriented Colorings}
\label{orientedColoring}

	In the rest of the paper, our main idea is to consider variants of {\sc proper k-coloring} by changing the coloring rules according to existing coloring graph parameters. Our study will generally be focused on paths and cycles. The first variant we investigate is a game based on oriented colorings, as described in \cite{Sop}.


\begin{definition}
        \label{defOrientedColoring}
	{\sc Oriented $k$-coloring} is an impartial coloring ruleset where for each arc $(u,v)$ in digraph $G$, the vertices $u$ and $v$ are not painted with the same color, and there is no other arc, $(v',u')$ of $E$ where $u$ is painted the same as $u'$ and $v$ the same as $v'$.
\end{definition}

	The oriented coloring of graphs has never been considered in the context of combinatorial games. One of the difficulties of {\sc Oriented $k$-coloring} is that the orientation of the colors is decided during the game (more precisely, when two adjacent vertices are colored for the first time). Hence we also choose to consider the variant of this game with $k=2$ and where the orientation is fixed before the first move:
	
\begin{definition}
        \label{defOrientedBlueRedColoring}
	{\sc Oriented Blue-Red-coloring} is an impartial coloring ruleset where $k=2$ (we will refer to the colors as Blue and Red) and for each arc $(u,v)$ of digraph, $G$: if both $u$ and $v$ are colored, then $u$ must be Blue, and $v$ must be Red.
\end{definition}

	This choice considerably simplifies the game on paths and cycles since coloring a vertex $v$ in Blue (resp. Red) forbids playing on the predecessor (resp. successor) of $v$. Hence coloring a vertex from an uncolored path $P$ is an operation that cuts $P$ into two shorter paths, with their extremities possibly colored.
	\begin{figure}[!ht]
		\label{grundyoriente}
		\centering
			\includegraphics[width=60mm]{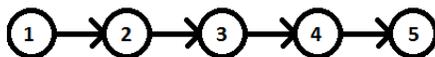}
		\caption{Example of oriented path of length 5}
	\end{figure}

	Thus, it is easy to remark that any game position of {\sc Oriented Blue-Red-coloring} on a path is a disjunctive sum of paths belonging to the four categories below:
\begin{itemize}
	\item The $A$ class : paths with the first vertex, $v_1$, colored Blue. Denote by $A_k$ a path of $A$ of length $k$.
	\item The $B$ class : paths with the last vertex, $v_k$, colored Red. Denote by $B_k$ a path of $B$ of length $k$.
	\item The $C$ class : paths with the first vertex, $v_1$, colored Blue, and the last, $v_k$, painted Red. Denote by $C_k$ a path of $C$ of length $k$.
	\item The $D$ class : paths with no vertex colored. Denote by $D_k$ a path of $D$ of length $k$.
\end{itemize}

We first give equivalences between the $A$ and $B$ classes, followed by a similar result concerning $C$ and $D$ classes.  In both cases, we let the vertices of a path of length $k$ be labeled from $v_1$ to $v_k$ : $v_1\rightarrow v_2 \rightarrow\ldots \rightarrow v_k$.

\begin{prop}
\label{PAB}
	For all integers $k>0: \G(A_k)=\G(B_k)$. 
\end{prop}

\begin{proof}
	Recursively, we prove that $\G(A_{k})=\G(B_{k})$ by showing that for any option of $A_{k}$, there exists an option of $B_{k}$ with the same Grundy value (note that every option on $B_{k}$ has a equivalent option in $A_{k}$ by a similar proof). \\
	Base Case: It is clear that $\G(A_{1})=\G(B_{1}) = 0$; there are no moves available.\\
	Suppose that for each $1\leq i<k$, we have $\G(A_{i})=\G(B_{i})$. We will now show that $\G(A_{k})=\G(B_{k}$).\\
	An option of $A_{k}$ is given by coloring the vertex $v_i$ ($2\leq i\leq k$) in red or blue. We prove that coloring $v_i$ is equivalent to coloring  $v_{k-i+1}$ on $B_{k}$ with the opposite color:
	\begin{itemize}
		\item If $v_i$ is colored Blue, then $A_{k}$ will be splitted into two shorter paths $A_{i-2}$ and $A_{k-(i-1)}$. Painting $v_{k-i+1}$ Red on $B_{k}$ splits into $B_{k-i+1}$ and $B_{i-2}$. Hence we conclude the hypothesis.
		\item If $v_i$ is colored Red, then $A_{k}$ will be splitted into two shorter paths $C_{i}$ and $D_{k-(i+2)}$. Painting $v_{k-i+1}$ Blue on $B_{k}$ splits into $D_{k-i-2}$ and $C_{i}$. 
	\end{itemize}

\end{proof}

\begin{prop}
\label{PCD}
	For all integers $k>0: \G(C_{k+3})=\G(D_{k})$.
\end{prop}

\begin{proof}
	We use the same technique as in the proof of Proposition \ref{PAB}.\invisibleComment{This was \ref{ACD}, but I think it should be PAB instead, is that correct?}\\
	Base Case ($k=1$): $\G(C_4) = \G(D_1) = 1$, because all options of either game have no further moves.\\
	By induction, we show that for any option of $C_{k+3}$, there exists an option of $D_{k}$ with the same Grundy value: playing Blue (resp. Red) on $v_i$ from $C_{k+3}$ is equivalent to playing Blue (resp. Red) on $v_{i-2}$ (resp. $v_{k-i+2}$) from $D_{k}$.
	Reciprocally, playing Blue (resp. Red) on $v_i$ from $D_{k}$ is equivalent to playing Blue (resp. Red) on $v_{i+2}$ (resp. $v_{k-i+2}$) from $C_{k+3}$.
\end{proof}

	In order to compute the Grundy function for the {\sc Oriented Blue-Red-coloring game} on uncolored paths (i.e., paths of the $D$ class), we remark that any move from a position in $A,B,C$ or $D$ leads to a sum of two paths also belonging to $A,B,C$ or $D$. This leads to the following recursive characterization:

\begin{prop}
	\label{ACD}

  
\begin{eqnarray*}
\G(D_{k})&=& \textit{Mex}\lbrace \\
	(1)& &\lbrace\G(D_{i-2})\oplus\G(A_{k+1-i}), \forall\; 1\leq i\leq k\rbrace \cup\\
	(2)& &\lbrace\G(B_{i})\oplus\G(D_{k-(i+1)}), \forall\; 1\leq i\leq k\rbrace \rbrace
\end{eqnarray*}

\begin{eqnarray*}
	\G(A_{k})&=&\textit{Mex}\lbrace \\
	(1)& &\{\G(A_{i-2})\oplus\G(A_{k+1-i}), \forall\; 3\leq i\leq k\} \cup\\
	(2)& &\{\G(C_{i})\oplus\G(D_{k-(i+1)}), \forall\; 2\leq i\leq k\}\rbrace\\
\end{eqnarray*}

\begin{eqnarray*}
	\G(C_{k})&=&\textit{Mex}\lbrace \\
	(1)& &\{\G(A_{i-2})\oplus\G(C_{k+1-i}), \forall\; 3\leq i\leq k-1\} \cup\\
	(2)& &\{\G(C_{i})\oplus\G(B_{k-(i+1)}), \forall\; 2\leq i\leq k-2\}\rbrace
\end{eqnarray*}

	
%
\end{prop}

\begin{proof}
	Visually. Example is given in Figure \ref{grundyoriente2}. In each Mex formula, the first part (1) corresponds to coloring $v_i$ with Blue, and the second part (2) corresponds to coloring $v_i$ with Red.
\end{proof}

	\begin{figure}[!ht]
		\centering
			\includegraphics[width=75mm]{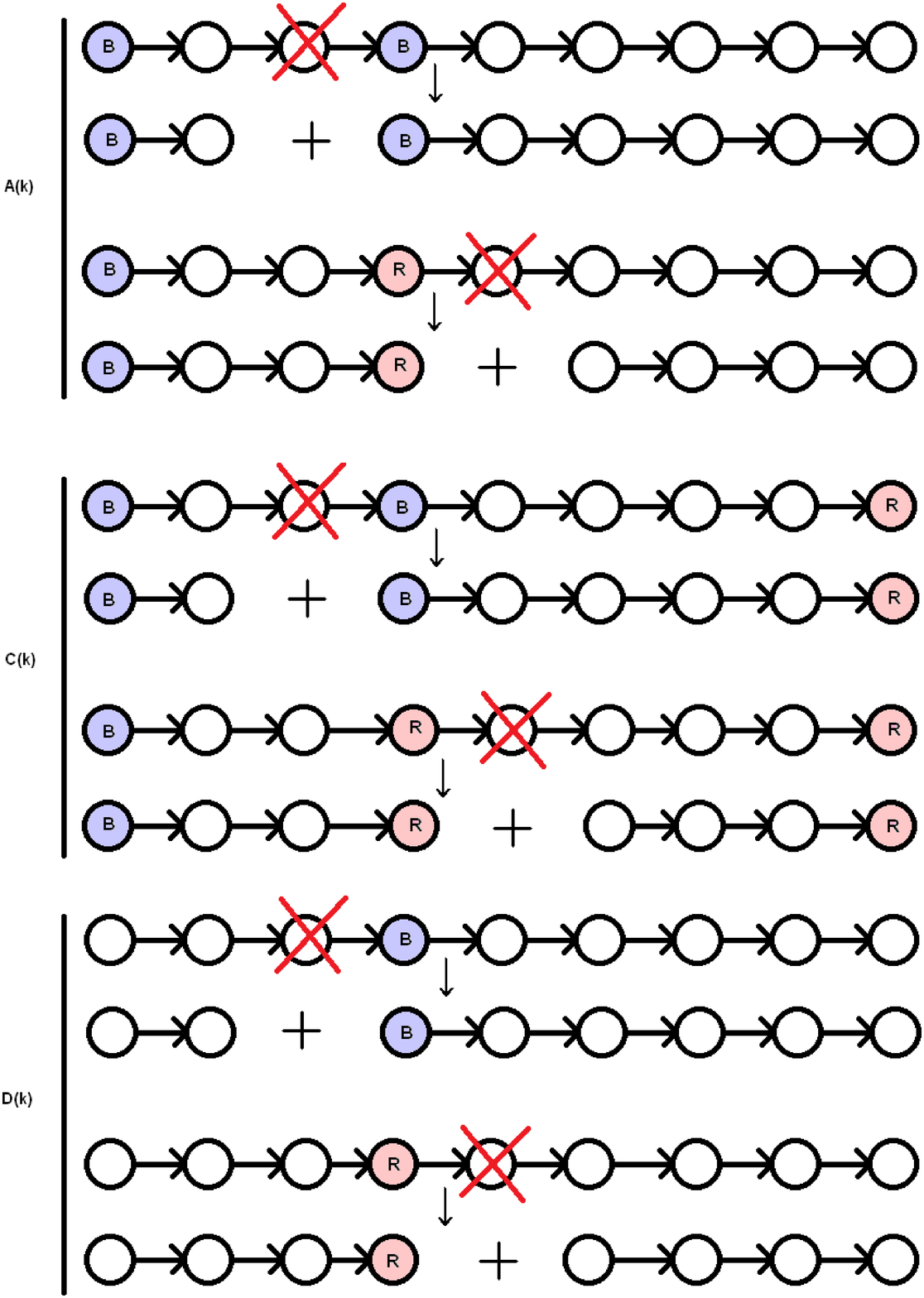}
		\caption{Proof of Proposition \ref{ACD}}
		\label{grundyoriente2}
	\end{figure}
	

	Unfortunately, the previous proposition does not provide a polynomial time algorithm to compute the value of $\G(D_k)$. Even the zeros of the Grundy function seem hard to characterize. By computing the first values of $\G$, we note that they split into rare and common values. This phenomenon, which is described in \cite{Win}, mainly appears for octal games and is frequently used to provide a quick computation of their nimbers. It appears in our study since the recursive characterization described in Proposition \ref{ACD} has the same structure as those of octal games.
	
The interest of such a partition of the nimbers is that the rare values form a closed space under nim-addition. The sum of two common values makes a rare value, while adding a rare and a common value gives a common value. In our situation, we observed that the sparse space is the closure (under the Nim-sum operator) of the set $\{0,1,2,3,4,5,6,7,24,40,64,136,264,520,1032\}$. By considering this property and together with Propositions \ref{PAB} and \ref{PCD}, we considerably improved the speed of the algorithm described in Poposition \ref{ACD}. In concrete terms, we computed the Grundy function of positions in $A,B,C$ and $D$ for paths of lengths $1$ up to $10 000 000$. Unfortunately, no periodicity appears in the Grundy function. However, we note the following facts:
\begin{itemize}
	\item $\forall k>3$, we have $\G(A_k)>0$ and $\G(B_k)>0$.
	\item We know $26$ $\PP$ positions for $D_k$, the last one being at $k=8084$.
	\item The maximal found value for $\G(D_k)$ is $\G(3099736)=1401$.
	\item The largest known index for a rare value is $k=642196$.
	\item There exist odd-length paths which are losing positions for the first player ($\PP$).
	\item $\G(D_k)$ is not periodic for $k<$10 000 000.
\end{itemize}
	
Although the Grundy function is not fully computed for the classes $A$ and $B$, we can prove that games played on these classes always are $\N$ positions:

\begin{prop}
	$\forall k>3: \G(A_k)>0$ and $\G(B_k)>0$.
\end{prop}
	
\begin{proof}
	Whatever the length of the path, it suffices to prove that there exists an option with Grundy value 0.\\
	Assume $k$ is odd. Then play Blue on $v_{(k+3)/2}$. According to Proposition \ref{ACD}, the Grundy value of the resulting position is equal to $\G(A_{(k+3)/2-2})\oplus \G(A_{k+1-(k+3)/2})=\G(A_{(k-1)/2})\oplus \G(A_{(k-1)/2})=0$. \\
	Assume $k$ is even. Then play Blue on $v_{k/2+1}$. The Grundy value is thus $\G(C_{k/2+1})\oplus \G(D_{k-(k/2+1+1)})=\G(C_{k/2+1})\oplus \G(D_{k/2-2})$. Since we know from Proposition \ref{PCD} that $\G(C_{k/2+1})=\G(D_{k/2-2})$, the Grundy value of this option is 0.
\end{proof}

\begin{corollary}
	Given an integer $l>3$, {\sc Oriented Blue-Red-coloring} starting positions on a cycle of length $l$ is in $\PP$.
\end{corollary}

\begin{proof}
	When first player plays Blue (resp. Red) on a node of the cycle, the predecessor (resp. successor) of the chain can not be painted. Then the chain is equivalent to an $A$ or $B$ case, which is always in $\N$. Therefore, the second player can win in all cases.
\end{proof}

\begin{open}
There are lot of questions that remain open for this kind of coloring:
\begin{itemize}
	\item Is $k=8084$ the last value for which $D_k$ is a P-position?
	\item Is there any periodicity of the Grundy function, with period > 10 000 000?
	\item Can we provide a polynomial time algorithm to compute the Grundy function?
	\item What if we consider that the orientation is not fixed before starting the game?
\end{itemize}
\end{open}

The computational complexity of this game mimics that of \textsc{Proper $k$-coloring}.  The proof is a corollary to Proposition \ref{properPSPACE}.

\begin{corollary}
  Determining whether a position of {\sc Oriented $k$-Coloring} is in $\N$ is \cclass{PSPACE}-complete.
\end{corollary}

This reduction is exactly the same as in Proposition \ref{properPSPACE} except that all edges are oriented.  

\begin{proof}
  Again, we reduce from {\sc Node-Kayles}.  Given a game of {\sc Node-Kayles} with undirected graph $G = (V, E)$, we construct a new directed graph, $G' = (V', E')$ as in the proof of Proposition \ref{properPSPACE} except that $E'$ consists of \emph{arcs}.  ($v_0$ has arcs leading to each $v_i$ in that reduction; each other edge can be directed in either direction.)
  In the game of {\sc Oriented $k$-Coloring} on $G'$, only one of the $k$ colors can be used by the players (since all vertices are adjacent to vertices of all other colors) and thus no two adjacent vertices can be colored, just as in the original {\sc Node-Kayles} game.  Thus the two games are equivalent and {\sc Oriented $k$-Coloring} is \cclass{PSPACE}-hard.  Since the game has a maximum number of moves, $\left|V'\right|$, the game is in \cclass{PSPACE}; together, deciding the outcome class is a \cclass{PSPACE}-complete problem.
\end{proof}

As it turns out, {\sc Oriented Blue-Red Coloring} is also a \cclass{PSPACE}-complete game.  

\begin{prop}
  Determining whether a position of {\sc Oriented Blue-Red Coloring} is in $\N$ is \cclass{PSPACE}-complete.
\end{prop}

This proof is based off the observation that if there are two vertices connected by arcs in both directions, only one of the pair can be colored.

\begin{proof}
  As with other proofs, we will reduce from {\sc Node-Kayles}.  Given a game of {\sc Node-Kayles} with undirected graph, $G = (V, E)$, we construct a new \emph{directed} graph, $G' = (V', E')$ where:

  $V' = V$ and
  
  $E' = (u, v), (v, u) : (u,v) \in E$.

  Since each pair of adjacent vertices is connected by arcs in both directions in $G'$, if a vertex is painted either Red or Blue, the adjacent vertices may not be colored at all.  Thus, each move on $G'$ is equivalent to the corresponding move on $G$ (ignoring the color).  Thus, the position of {\sc Oriented Blue-Red Coloring} on $G'$ is equivalent to the {\sc Node-Kayles} game on $G$ and the game is \cclass{PSPACE}-hard.  Since the maximum number of moves is $\left\vert V'\right \vert$, the solution can be computed in \cclass{PSPACE} and deciding the outcome class is a \cclass{PSPACE}-complete problem.
\end{proof}

\section{Weak Colorings}
\label{weakColoring}

	We now consider a coloring game on a graph $G$ with the consraint that each game position is a partial weak coloring of $G$. A weak coloring of a graph $G$ is a vertex coloring such that each non-isolated vertex is adjacent to at least one vertex with different color \cite{Chew}.

\begin{definition}
        \label{defWeakColoring}
	{\sc Weak 2-coloring} is an impartial coloring ruleset where two adjacent vertices can be colored with the same color if and only if each vertex is also adjacent to a vertex of the opposite color.
\end{definition}

	In the field of chromatic graph theory, weak coloring is completely solved, in particular since each graph is 2-weakly colorable. \\
	
	We here study {\sc Weak 2-coloring} on paths and cycles. In the case of paths and cycles of even length, we can use a generalization of Propositions \ref{gagnant} and \ref{perdant} to conclude that these instances of {\sc Weak 2-coloring} are all $\PP$ positions, while odd paths are $\N$ positions.

Alternative consideration must be applied in the case of odd cycles, which differs from the results for {\sc Proper 2-coloring}.
	
\begin{theorem}\label{weak}
	{\sc Weak 2-coloring} starting positions on odd cycles are in $\N$.
\end{theorem}
	
\begin{proof}

For a positive integer $k$, we consider a cycle of length $2k+1=n$. By way of contradiction, suppose there is a winning strategy for the second player. As both players play alternately, there must be at least one vertex, say $v$, which is not colored at the end of game. Denote by $l_0$ and $r_0$ the two adjacent vertices of $v$, and by $l_1,r_1$ the other adjacent vertices to $l_0$ and $r_0$ respectively. By the same way, we denote by $l_2\ldots l_k$ (resp. $r_2\ldots r_k$) the other adjacent vertices of $l_1\ldots l_{k-1}$ (resp. $r_1\ldots r_{k-1}$). We first remark that at the end of the game, $l_{0}$ and $r_{0}$ hold necessarily different colors. If not, $v$ could have been colored, contradicting the hypothesis.
	If $l_1$ (resp. $r_1$) has the opposite color of $l_0$ (resp. $r_0$), we are able to color $v$ in the same color as $l_0$ (resp. $r_0$) contradicting the hypothesis. In addition, since $v$ is not colored, $l_0$ and $r_0$ can not be adjacent to a vertex of the same color. Therefore, $l_{1}$ and $r_{1}$ are uncolored at the end. 
Moreover, we have $l_{1}\neq r_{1}$ since the cycle has an odd length.
	By repeating these steps, we show the fact $(F)$ : the vertices $l_{2i+1}$ and $r_{2i+1}$ for $0\leq i < k/2$ are not colored, while the vertices $l_{2i}$ and $r_{2i}$ for $1\leq i < k/2$ have opposite colors. We now consider the final adjacent vertices $l_k$ and $r_k$:
	\begin{itemize}
		\item If $k$ is odd, there are two uncolored adjacent vertices $l_k$ and $r_k$ at the end of the game, which is not possible.
		\item If $k$ is even, there are two colored adjacent vertices $l_k$ and $r_k$ with different colors. Hence $l_{k-1}$ and $r_{k-1}$ can be colored, contradicting the fact $(F)$.
	\end{itemize}
	Consequently, there is no winning strategy for the second player.
\end{proof}
	
\begin{corollary}
	{\sc Weak 2-coloring} starting positions on odd cycles have grundy value $g$ where $g$ is the parity of the number of uncolored vertices.  Thus, the strategies of the players does not matter in the outcome of the game.
\end{corollary}

\begin{proof}
	According to the proof of Theorem \ref{weak}, each vertex is colored at the end of the game, independent of the moves of both players. 
\end{proof}

	Hence, {\sc Weak 2-coloring} is completely solved on paths and cycles. We did not investigate this game on other kinds of graphs, but we think that playing on any graph may not be as hard as the previous problems. This consideration is due to the fact that every graph can be weakly colored with two colors. \\
	
\begin{open}
~\begin{itemize}
\item What is the computational complexity of {\sc Weak 2-coloring} on any graph? Is it polynomial?
\item Does the problem become straightforward if there is more than two available colors?
\end{itemize}
\end{open}

\section{2-distance Colorings}
\label{distanceColoring}

	In this section we study a kind of graph coloring called {\it 2-distance coloring}. It has been previously studied in \cite{Fer}. As before, we study it in the context of impartial combinatorial games.

\begin{definition}
        \label{defDistanceColoring}
	{\sc $d$-distance $k$-coloring game} is an impartial coloring ruleset where two vertices at distance less than or equal to $d$ cannot be painted the same color.
\end{definition}

	Our results focus mainly on the case where $d = 2$.
	
\begin{theorem}
	{\sc 2-distance 2-coloring} starting positions are in $\PP$ when $G$ is an even length path or an even length cycle.
\end{theorem}

\begin{proof}
	Playing {\sc 2-distance 2-coloring} on a graph $G=(V,E)$ is equivalent to playing {\sc proper 2-coloring} on $G'=(V,E')$ where $E'=E\cup \lbrace (u,v):d(u,v)=2\rbrace$. The notation $d(u,v)$ corresponds to the (edge)-distance between two vertices. When $G$ is a path or a cycle, by applying the same proof as for Proposition \ref{perdant} on $G'$, we can conclude it is in $\PP$. Indeed, if we label the vertex of $G'$ from $1$ to $2k$, the involution that we consider is $s(i)=2k-i+1$.
\end{proof}

	The case of odd length paths and cycles remains open in the general case. We implemented a recursive algorithm to compute the outcome class ($\PP$ or $\N$) of odd paths for lengths between 3 and 17. We did not manage to compute it for higher lengths, beacause of the complexity of the algorithm. Here are the first results:
\begin{center}
	\begin{tabular}{|c|c|}
  	\hline
  		Length & Status \\ \hline
  		3 & $\PP$ \\ \hline
  		5 & $\N$ \\ \hline
  		7 & $\N$ \\ \hline
  		9 & $\PP$ \\ \hline
  		11 & $\PP$ \\ \hline
  		13 & $\N$ \\ \hline
  		15 & $\PP$ \\ \hline
  		17 & $\PP$ \\ \hline
	\end{tabular}
\end{center}

	In order to simplify the computation, we prove the following property :
	
\begin{prop}
	If there is a winning strategy for the first player in {\sc 2-distance 2-coloring} on odd paths, the central vertex must be colored by the first player.
\end{prop}

\begin{proof}
	Vertices are labeled from $1$ to $2k+1$. Let $s$ be the involution $s(i)=2k+1-i$. Suppose that the first player can win without coloring the middle vertex. For each vertex $i$ colored by the first player, the second player can color $s(i)$ with the other color. Finally the second player will win the game.
\end{proof}

	Hence we cannot conclude the general case of {\sc 2-distance 2-coloring} on odd length paths, but we solve the other cases and the odd lengths from 1 to 17. This game is complex to study, particularly since it does not simplify into a sum of two smaller games. 	

  Another indication for the complexity of this game is the computational hardness of the general version.  We now show that the game is \cclass{PSPACE}-complete in general

\begin{prop}
  \label{distancePSPACE}
  Determining whether a position in {\sc 2-Distance 2-coloring} is in $\N$ is \cclass{PSPACE}-complete.
\end{prop}

\begin{proof}
  Since all instances of the game are short and the number of plays is bounded above by the number of unpainted vertices, the problem is in \cclass{PSPACE}.  As with the other hardness proofs in this paper, we will reduce from \textsc{Node-Kayles}.  Again, the gadgets will restrict players to only have one color available, and also separate each original pair of adjacent vertices by an extra unpaintable vertex.

  Given an instance of \textsc{Node-Kayles}, $G = (V, E)$, create a new graph, $G' = (V', E')$ where 

$V' = V \cup \{v_{e, 0}, v_{e,1}, v_{e,2}: \forall e \in E\}$ and

$E' = \{(u, v_{e, 0}), (w, v_{e, 0}), (v_{e, 0}, v_{e, 2}), (v_{e, 2}, v_{e, 1}) : \forall (u, w) = e \in E\}$.

  \begin{figure}[!ht]
	  \label{fig:2Distance2ColoringReduction}
	  \centering
		  \includegraphics[width=60mm]{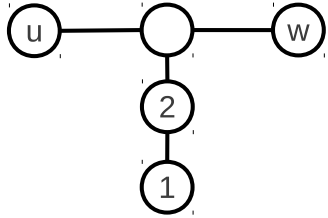}
	  \caption{Reduction of a \textsc{Node-Kayles} edge $(u, w)$ to {\sc 2-Distance 2-Coloring}.  Numbers indicate colors for vertices.}
  \end{figure}

  To complete the reduction, we paint all vertices $v_{e, 1}$ with color 1 and all vertices $v_{e, 2}$ with color 2.  In the resulting graph, \emph{all} vertices are within distance 2 of a vertex painted with color 2, and all vertices not in $V$ are within distance 2 of a vertex pained with color 1.  Thus, only vertices in $V$ can be painted (and only with color 1).  These are all within distance 2 of another vertex from $V$ exactly when those vertices were adjacent in $G$.  Thus, each move in this game is equivalent to the corresponding move on the \textsc{Node-Kayles} instance, meaning this game is also \cclass{PSPACE}-hard.

  Thus, determining whether a position of {\sc 2-Distance 2-Coloring} is in $\N$ is \cclass{PSPACE}-complete.
\end{proof}

We can expand on this by providing for any number of colors.

\begin{corollary}
 {\sc 2-Distance k-Coloring} is \cclass{PSPACE}-complete for any number of colors, $k$.
\end{corollary}

\begin{proof}
 For each additional color, add another node painted that color adjacent to $v_{e, 0}$ for each edge $e \in E$.  Now that color is not available to paint any vertices on the new graph.
\end{proof}

\begin{open}
Among the different games listed in this paper, this ruleset is certainly the hardest one to solve. The investigation of odd paths remains for us paramount to go further on it.
\end{open}

\section{Sequential Coloring (also called Online Coloring)}
\label{sequentialColoring}

	A sequential coloring is based on assigning colors to vertices in some predetermined order.  Games based on sequential colorings were introduced by Bodlaender \cite{Bod} in 1989, where he presented very different results for the separate cases $k=2$ and $k \geq 3$.  To the best of our knowledge, the sequential coloring game is the only variant that has been previously considered in the literature on impartial coloring games.

\begin{definition}
        \label{defSequentialColoring}
	{\sc Sequential $k$-coloring} is an impartial coloring ruleset which includes a sequential function, $f:\{1, \ldots, \left|V\right|\} \mapsto V$.  On the $i$-th turn, the current player \emph{must} properly paint the vertex $f(i)$.  In other words, the only legal moves must paint the uncolored vertex with the lowest index, and this coloring must remain proper; no neighboring vertices may be painted the same color.
\end{definition}

	For $k \geq 3$, Bodlaender \cite{Bod} finds the game to be computationally hard.
	
\begin{prop}
	For $k \geq 3$, it is \cclass{PSPACE}-complete to determine whether an position of \textsc{Sequential $k$-coloring} is in $\N$ \cite{Bod}.
\end{prop}

	The remainder of this section considers the case where $k = 2$.  In this case, Bodlaender presents a polynomial-time algorithm to solve the game on any graph. The complexity of his algorithm is about $O(n+e\alpha(e,n))$ where $\alpha$ is the inverse Ackermann function, and $n$ the order of the graph.

	According to this definition, player 1 colors the vertices $f(2i+1)$, and player 2 the vertices $f(2i)$ for $i\geq 0$. Although Bodlaender's algorithm is very efficient, we here propose a better algorithm for paths and cycles, of complexity $O(n)$. Let $P_{n}$ be the path of length $n$, with the vertices labeled from $1$ to $n$, and consider {\sc sequential 2-coloring} on $P_{n}$, with a given sequential function $f$. We denote $g=f^{-1}$.\\

\begin{definition}
	We classify the set of vertices of $P_n$ into three types :
	\begin{itemize}
		\item Source vertices: these are local minima of f. $v$ is a source vertex if and only if $g(v)<g(v_{1})$ and $g(v)<g(v_{2})$, where $v_1$ and $v_2$ are adjacent vertices of $v$. Note that if $v$ is an extremity of $P_n$, $v_2$ does not exist and only the first condition is considered. Roughly speaking, source vertices are those that are colored  before their neighbours.
		\item Closed vertices: these are local maxima of f. Let $v_{1}$ and $v_{2}$ be adjacent vertices of $v$. $v$ is a closed vertex if and only if $g(v)>g(v_{1})$ and $g(v)>g(v_{2})$. Roughly speaking, if $v_1$ and $v_2$ have different colors, the player who colors a closed vertex $v$ loses the game.
		\item Constrained vertices: these are neither maxima or minima of f. Let $v_{1}$ and $v_{2}$ be two adjacent vertices of $v$ (without loss of generality, we suppose $g(v_{2})>g(v_{1})$). $v$ is a constrained vertex if and only if $g(v_{2})>f(v)>g(v_{1})$.
	\end{itemize}
\end{definition}
	
\begin{prop}
	\label{choice}
	A player has to make a color choice only on source vertices. He may lose only when playing on closed vertices.
\end{prop}

\begin{proof}
	For any vertex $v$ which is not a source, at least one adjacent vertex of $v$ has been previously painted. On constrained vertices, there is exactly one vertex colored, which means the color of $v$ is imposed. If $v$ is a closed vertex, either $v$ is not colorable (if the two adjacent vertices are holding different colors), or the color of $v$ is imposed.
\end{proof}

\begin{theorem}
	Given a path $P$ of length $n$ and a sequential function $f$, we can decide in time $O(n)$ whether this position is in $\PP$ or $\N$.
\end{theorem}

\begin{proof}
	According to Proposition \ref{choice}, a player may lose only when he must play on a closed vertex, and he makes choices only on source vertices. This explains why we will consider only source and closed vertices, and we can remove the constrained ones from the path, by connecting their two adjacent vertices. For example, let us consider the game on a path $P$ of length 11, using the sequential function $f$ below (recall that $f^{-1}=g$):
	\begin{figure}[!ht]
		\label{constrained}
		\centering
			\includegraphics[width=90mm]{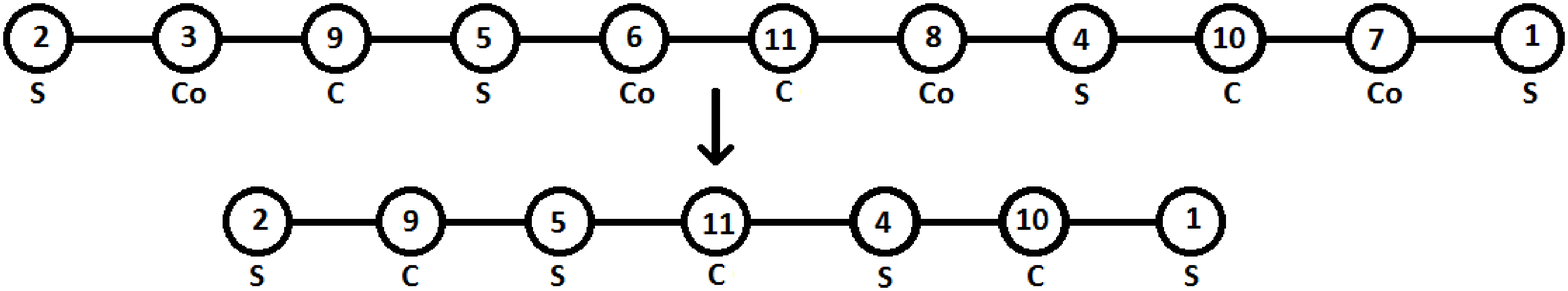}
		\caption{Removing constrained vertex. Labels on vertices are their value using $f^{-1}$, S= Source, C = Closed, Co = Constrained}
	\end{figure}
	We are now looking for the player who will lose. Hence we have to consider closed vertices, ordered by increasing values of $f$. On the above example, the first closed vertex $u_0$ has the value 9 according to $f$. Since 9 is odd, the first player will color it. Consider its two adjacent vertices $v_0$ and $v_1$, with respectively the values 2 and 5. As $v_0$ will be colored first, the first player will be able to color $u_0$ if and only if he owns $v_1$. If he does not, the other player can color $v_0$ and $v_1$ in a such way the first player will lose on $u_0$(by coloring $v_0$ and $v_1$ with the same color).

As $f(v_1)=5$, the first player owns $v_1$ and he will be able to color $u_0$ by choosing the appropriate color on $v_1$.
	Consequently, the first player does not have the choice when coloring $v_1$ and $u_0$, and we can delete them as if they were constrained vertices:
	
	\begin{figure}[!ht]
		\label{constrained2}
		\centering
			\includegraphics[width=75mm]{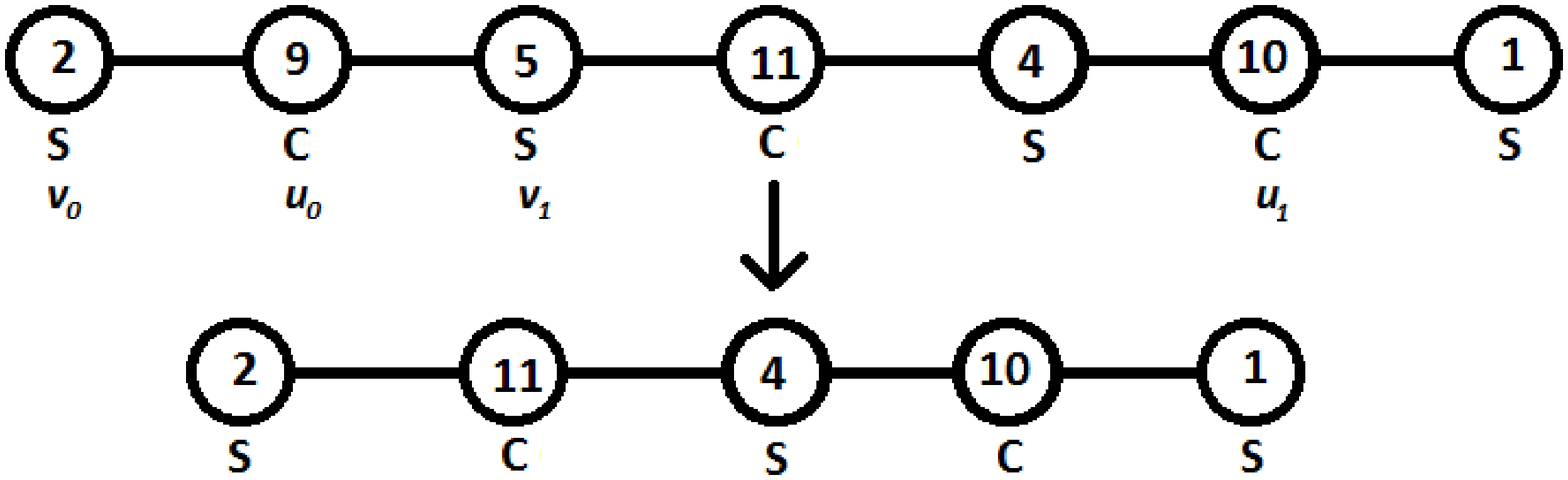}
		\caption{Removing vertices 9 and 5, since the first player has no choice.}
	\end{figure}
	
	By considering the next closed vertex ($u_1$ with the value $10$), and with the same argument, as the second player owns the vertices 10($v_3$) and 4($v_2$), we also delete them.\\
	Then, when considering the last vertices 11, 1 and 2 (respectively $u_2$,$v_4$,$v_5$ on the figure), one can see that the parity of $v_5$ is not the same as $u_2$, hence the owner of the vertices is not the same. Therefore, $u_2$ can not be colored by the first player, who will lose the game.\\
	
	The algorithm to decide who the winner is has five steps:
	\begin{itemize}
		\item Classify the vertices into source, constrained and closed.
		\item Remove the constrained vertices.
		\item Order the closed vertices by increasing values of the $f$ function, and pick them up in this order.
		\item On each closed vertex $u$ picked up, analyse the two adjacents source vertices. If the owner of the vertex $v$ with the higher value of $f$ is the same as $u$, then remove $u$ and $v$ and repeat this operation. Otherwise the game is losing for the owner of $u$.
		\item If there are no more vertices left, and no result yet, the game is winning for the first player if the length of the path is odd. Otherwise it is winning for the second player.
	\end{itemize}
	
	Each step take $O(n)$ operations, especially the ordering because we know the maximum value of $f$. 
\end{proof}

	
	
\section{Conclusions}
        
        This work contains many new results concerning impartial coloring games as alternatives to the strictly partisan graph coloring games of \textsc{Col} and \textsc{Snort}.  We define six new rulesets concerning games that can use any number of colors, two of which, \textsc{Oriented $k$-coloring} (definition \ref{defOrientedColoring}) and \textsc{Oriented Blue-Red-coloring} (definition \ref{defOrientedBlueRedColoring}) are specific to directed graphs.  As far as the authors are aware, only \textsc{Sequential $k$-coloring} has been studied prior to this work.
        
        For many of the rulesets, we study the winnability of specific graph types. For paths, we show how to find the outcome class for starting positions of \textsc{Proper 2-coloring} and for even-length paths in \textsc{2-distance 2-coloring}.  
        
        For cycles, we solve the starting positions for \textsc{Proper 2-coloring}, \textsc{Oriented Blue-Red-coloring}, odd cycles for \textsc{Weak 2-coloring} and even cycles for \textsc{2-distance 2-coloring}.  The result for \textsc{Weak 2-coloring} determines the outcome class for all of positions of odd cycles, not just starting positions.  Furthermore, the game is trivial in this case: the grundy values are all either 0 or 1, meaning that player strategies will not affect the outcome of the game.
        
        For those graphs that have an involution, we show cases which can determine the outcome class of starting positions of \textsc{Proper $k$-coloring}.
        
        Concerning the computational complexity of determining the outcome class, we show that four of these games are \cclass{PSPACE}-complete: \textsc{Proper $k$-coloring}, \textsc{Oriented $k$-coloring}, \textsc{Oriented Blue-Red-coloring} and \textsc{$d$-Distance $k$-coloring} for the case where $d=2$.  For \textsc{Sequential $2$-coloring}, we improve upon a previous result by finding a new algorithm to determine the outcome class of the game in $O(n)$ time.  The complexities of \textsc{$d$-Distance $k$-coloring} for $d>2$ and \textsc{Weak $k$-coloring} remain open.






\end{document}